\documentclass{article}
\title{\textbf{
Some orbits of a two-vertex stabilizer\\in a Grassmann graph}}
\usepackage{authblk}
\author[]{Ian Seong}
\date{}
\usepackage {parskip}
\usepackage [margin=2.5cm] {geometry}
\usepackage{enumitem}
\usepackage[utf8]{inputenc}
\usepackage[T1]{fontenc}
\usepackage{geometry}
\usepackage{amsmath, amssymb, amsthm}
\usepackage{hyperref}
\usepackage{graphicx}
\usepackage{mathtools}
\usepackage{relsize}
\usepackage{bm}
\usepackage{esvect}
\usepackage{tikz}
\usepackage{hhline}
\usepackage{enumitem}
\usepackage{comment}
\usepackage{ytableau}
\usepackage{pdflscape}
\usepackage{mathabx}

\def\Stab{{\rm{Stab}}}
\def\Fix{{\rm{Fix}}}
\newtheoremstyle{dotless}{}{}{\itshape}{}{\bfseries}{}{ }{}
\theoremstyle{dotless}
\newtheorem{theorem}{Theorem}[section]

\newtheorem{lemma}[theorem]{Lemma}

\newtheoremstyle{dotlessdef}{}{}{}{}{\bfseries}{}{ }{}
\theoremstyle{dotlessdef}
\newtheorem{definition}[theorem]{Definition}
\newtheorem{remark}[theorem]{Remark}

\begin{document}

\maketitle
\begin{abstract}
   Let $\mathbb{F}_q$ denote a finite field with $q$ elements. Let $n,k$ denote integers with $n>2k\geq 6$. Let $V$ denote a vector space over $\mathbb{F}_{q}$ that has dimension $n$. The vertex set of the Grassmann graph $J_q(n,k)$ consists of the $k$-dimensional subspaces of $V$. Two vertices of $J_q(n,k)$ are adjacent whenever their intersection has dimension $k-1$. Let $\partial$ denote the path-length distance function of $J_q(n,k)$. Pick vertices $x,y$ of $J_q(n,k)$ such that $1<\partial(x,y)<k$. Let $\Stab(x,y)$ denote the subgroup of $GL(V)$ that stabilizes both $x$ and $y$. In this paper, we investigate the orbits of $\Stab(x,y)$ acting on the local graph $\Gamma(x)$. We show that there are five orbits. By construction, these five orbits give an equitable partition of $\Gamma(x)$; we find the corresponding structure constants. In order to describe the five orbits more deeply, we bring in a Euclidean representation of $J_q(n,k)$ associated with the second largest eigenvalue of $J_q(n,k)$. By construction, for each orbit its characteristic vector is represented by a vector in the associated Euclidean space. We compute many inner products and linear dependencies involving the five representing vectors. 
\end{abstract}

\section{Introduction}
This paper is about a family of finite undirected graphs known as
the Grassmann graphs \cite[p.~268]{BCN}. These graphs are distance-regular in the
sense of \cite[p.~1]{BCN}. For any distance-regular graph, there is a construction called a Euclidean representation. In order to
motivate our main topic, we now recall this construction. 
Let $\Gamma$ denote a distance-regular graph with vertex set $X$ and path-length distance function $\partial$. According to \cite[Definition~6.1]{Seong}, a Euclidean representation of $\Gamma$ is a nonzero Euclidean space $E$ together with a map $\rho:X\rightarrow E$ such that
\begin{enumerate}[label=(\roman*)]
    \item $E$ is spanned by $\{\rho(x)\mid x\in X\}$;
    
    \item for all $x,y\in X$, the inner product $\bigl<\rho(x),\rho(y)\bigr>$ depends only on $\partial(x,y)$;
    
    \item there exists $\vartheta\in \mathbb{R}$ such that for all $x\in X$,
    \begin{equation*}
        \sum_{\substack{z\in X\\\partial(z,x)=1}}\rho(z)=\vartheta \rho(x).
    \end{equation*} 
\end{enumerate}
By \cite[Section~6]{Seong}, the scalar $\vartheta$ is an eigenvalue of $\Gamma$. For each eigenvalue $\theta$ of $\Gamma$, the corresponding eigenspace gives a Euclidean representation of $\Gamma$.

We briefly recall the definition of a Grassmann graph. Let $\mathbb{F}_{q}$ denote a finite field with $q$ elements. Fix an integer $n\geq 1$. Let $V$ denote a vector space over $\mathbb{F}_{q}$ that has dimension $n$. Let the set $P_q(n)$ consist of the subspaces of $V$. For $0\leq k\leq n$ let the set $P_k$ consist of the elements of $P_q(n)$ that have dimension $k$. For $1\leq k\leq n-1$ the vertex set of the Grassmann graph $J_q(n,k)$ is $P_k$. Two vertices of $J_q(n,k)$ are adjacent whenever their intersection has dimension $k-1$.

For the rest of this section, we assume that $\Gamma$ is the Grassmann graph $J_q(n,k)$ with $n>2k\geq 6$.

In what follows, we will use the notation 
\begin{equation*}
    [m]=\frac{q^{m}-1}{q-1} \qquad \qquad(m\in \mathbb{Z}).
\end{equation*} 
By \cite[Theorem~9.3.3]{BCN}, the eigenvalues of $\Gamma$ are:
\begin{equation*}
    \theta_i=q^{i+1}[k-i][n-k-i]-[i] \qquad \qquad (0\leq i\leq k).
\end{equation*}

In \cite[Section~4]{Seong}, we used $P_q(n)$ to construct a Euclidean representation of $\Gamma$ associated with $\theta_1$. We now recall this construction. Let $E$ denote a Euclidean space with dimension $[n]-1$ and bilinear form $\left<\; ,\;\right>$. Define a function 
\begin{equation}
    \begin{aligned}
    \label{introbasichat}
        P_1&\rightarrow E\\s&\mapsto\widehat{s}
    \end{aligned}
\end{equation}
that satisfies the following conditions (C1) $-$ (C4):

\begin{enumerate}[label=(C\arabic*)]
    \item $E=\text{Span}\bigl\{\widehat{s}\mid s\in P_1\bigr\}$;
    
    \item for $s\in P_1$, $\bigl\Vert\widehat{s}\bigr\Vert^2=[n]-1$;
    
    \item for distinct $s,t\in P_1$, $\Bigl<\widehat{s},\widehat{t}\Bigr>=-1$;
    
    \item $\mathlarger{\sum_{s\in P_1}\widehat{s}=0}$.
\end{enumerate}

Next, extend the function (\ref{introbasichat}) to a function
\begin{equation}
    \begin{aligned}
    \label{introhardhat}
        P_q(n)&\rightarrow E \\
        u&\mapsto \widehat{u}
    \end{aligned}
\end{equation}
such that for all $u\in P_q(n)$, 
\begin{equation*}
    \widehat{u}=\sum_{\substack{s\in P_1\\ s\subseteq u}}\widehat{s}.
\end{equation*}
By \cite[Section~6]{Seong}, the Euclidean space $E$, together with the restriction of the map (\ref{introhardhat}) to $X=P_k$, gives a Euclidean representation of $\Gamma$ that is associated with $\theta_1$. By \cite[Lemma~4.2]{Seong}, the $GL(V)$-action on $P_q(n)$ induces a $GL(V)$-module structure on $E$.

We now summarize the results of \cite{Seong}. For $x\in X$, let $\Gamma(x)$ denote the local graph of $x$. For the rest of the section, fix $x,y\in X$ such that $1<\partial(x,y)<k$. 
Let $\Stab(x,y)$ denote the subgroup of $GL(V)$ that stabilizes both $x$ and $y$. Let $\Fix(x,y)$ denote the subspace of $E$ consisting of the vectors that are fixed by every element of $\Stab(x,y)$. In \cite[Lemma~8.3]{Seong}, we showed that the following vectors form a basis for $\Fix(x,y)$:
\begin{equation}
\label{introbasis1}
    \widehat{x},\qquad \qquad \widehat{y}, \qquad \qquad \widehat{x\cap y}, \qquad \qquad \widehat{x+y}.
\end{equation}

We now describe a second basis for $\Fix(x,y)$. In \cite[Definition~9.1]{Seong}, we defined the sets 
\begin{equation*}
    \mathcal{B}_{xy}=\{z\in \Gamma(x)\mid \partial(z,y)=\partial(x,y)+1\}, \qquad \qquad \mathcal{C}_{xy}=\{z\in \Gamma(x)\mid \partial(z,y)=\partial(x,y)-1\}.
\end{equation*}
By \cite[Lemma~9.2,~9.4]{Seong}, the sets $\mathcal{B}_{xy},\mathcal{C}_{xy}$ are orbits of the $\Stab(x,y)$-action on $\Gamma(x)$. In \cite[Definition~9.5]{Seong}, we defined the vectors
\begin{equation*}
    B_{xy}=\sum_{z\in \mathcal{B}_{xy}}\widehat{z}, \qquad \qquad C_{xy}=\sum_{z\in \mathcal{C}_{xy}}\widehat{z}.
\end{equation*}

In \cite[Theorem~11.1]{Seong}, we showed that the following vectors form a basis for $\Fix(x,y)$:
\begin{equation}
\label{introbasis2}
    \widehat{x},\qquad \qquad \widehat{y}, \qquad \qquad B_{xy}, \qquad \qquad C_{xy}.
\end{equation}

In \cite[Theorem~11.3]{Seong}, we found the transition matrices between the basis (\ref{introbasis1}) and the basis (\ref{introbasis2}). We found the inner products between:
\begin{enumerate}[label=(\roman*)]
    \item any pair of vectors in the basis (\ref{introbasis1}) \cite[Theorem~10.4]{Seong};
    \item any pair of vectors in the basis (\ref{introbasis2}) \cite[Theorem~10.15]{Seong};
    \item any vector in the basis (\ref{introbasis1}) and any vector in the basis (\ref{introbasis2}) \cite[Theorem~10.9]{Seong}.
\end{enumerate}

In this paper, we investigate the orbits of $\Stab(x,y)$ acting on $\Gamma(x)$. As we will see, there are five orbits. We already mentioned two of the orbits, namely $\mathcal{B}_{xy}$ and $\mathcal{C}_{xy}$. We now describe the other three orbits. 

Define the set
\begin{equation*}
    \mathcal{A}_{xy}=\{z\in \Gamma(x)\mid \partial(y,z)=\partial(x,y)\}.
\end{equation*}

We partition the set $\mathcal{A}_{xy}$ into the following three sets: 
\begin{align*}
    \mathcal{A}_{xy}^{+}&=\{z\in \mathcal{A}_{xy}\mid z+x+y\supsetneq x+y,\;z\cap x\cap y=x\cap y\},\\
    \mathcal{A}_{xy}^{0}&=\{z\in \mathcal{A}_{xy}\mid z+x+y=x+y,\;z\cap x\cap y=x\cap y\},\\
    \mathcal{A}_{xy}^{-}&=\{z\in \mathcal{A}_{xy}\mid z+x+y=x+y,\;z\cap x\cap y\subsetneq x\cap y\}.
\end{align*}

We show that the sets 
\begin{equation*}
    \mathcal{A}_{xy}^{+}, \qquad \qquad \mathcal{A}_{xy}^{0}, \qquad \qquad \mathcal{A}_{xy}^{-}
\end{equation*}
are orbits of the $\Stab(x,y)$-action on $\Gamma(x)$. Hence,
\begin{equation}
\label{intro5orbit}
    \mathcal{B}_{xy},\qquad \qquad 
    \mathcal{C}_{xy},\qquad \qquad 
    \mathcal{A}_{xy}^{+}, \qquad \qquad \mathcal{A}_{xy}^{0}, \qquad \qquad \mathcal{A}_{xy}^{-}
\end{equation}
are the five orbits of the $\Stab(x,y)$-action on $\Gamma(x)$. By construction, (\ref{intro5orbit}) is a partition of $\Gamma(x)$ that is equitable in the sense of \cite[p.~159]{Schwenk}. We call this partition the $y$-partition of $\Gamma(x)$. 

Define the vectors
\begin{equation}
\label{introadef}
    A_{xy}^{+}=\sum_{z\in \mathcal{A}_{xy}^{+}}\widehat{z},\qquad \qquad A_{xy}^{0}=\sum_{z\in \mathcal{A}_{xy}^{0}}\widehat{z},\qquad \qquad A_{xy}^{-}=\sum_{z\in \mathcal{A}_{xy}^{-}}\widehat{z}.
\end{equation}

We show that $A_{xy}^{+},A_{xy}^{0},A_{xy}^{-}$ are contained in $\Fix(x,y)$. We write each vector in (\ref{introadef}) as a linear combination of the vectors in (\ref{introbasis1}) and also the vectors in (\ref{introbasis2}). We find the inner products between:
\begin{enumerate}[label=(\roman*)]
    \item any vector in (\ref{introadef}) and any vector in the basis (\ref{introbasis1});
    
    \item any vector in (\ref{introadef}) and any vector in the basis (\ref{introbasis2});

    \item any pair of vectors in (\ref{introadef}).
\end{enumerate}

We mentioned that the $y$-partition of $\Gamma(x)$ is equitable. We compute the corresponding structure constants. In the table below, for each orbit $\mathcal{O}$ in the header column, and each orbit $\mathcal{N}$ in the header row, the $(\mathcal{O},\mathcal{N})$-entry gives the number of vertices in $\mathcal{N}$ that are adjacent to a given vertex in $\mathcal{O}$. Write $i=\partial(x,y)$.

\begin{center}
\begin{tabular}{c|c c c c c}
     & $\mathcal{B}_{xy}$ & $\mathcal{C}_{xy}$ & $\mathcal{A}_{xy}^{+}$ & $\mathcal{A}_{xy}^{0}$ & $\mathcal{A}_{xy}^{-}$\\
    \hline \\
    $\mathcal{B}_{xy}$ & $\substack{q^{i+1}[k-i]\\+q^{i+1}[n-k-i]-q-1}$ & $0$ & $q[i]$ & $0$ & $q[i]$\\
   \\
    $\mathcal{C}_{xy}$ & $0$ & $2q[i-1]$ & $q^{i+1}[n-k-i]$ & $(q-1)\bigl(2[i]-1\bigr)$ & $q^{i+1}[k-i]$\\
   \\
    $\mathcal{A}_{xy}^{+}$ & $q^{i+1}[k-i]$ & $[i]$ & $q[n-k]-q-1$ & $(q-1)[i]$ & $0$\\
    \\
    $\mathcal{A}_{xy}^{0}$ & $0$ & $2[i]-1$ & $q^{i+1}[n-k-i]$ & $(q-1)\bigl(2[i]-1\bigr)-1$ &$q^{i+1}[k-i]$\\
    \\
    $\mathcal{A}_{xy}^{-}$ & $q^{i+1}[n-k-i]$ & $[i]$ & $0$ & $(q-1)[i]$ & $q[k]-q-1$\\
    \\
\end{tabular}
\end{center}

Let $\mathcal{M}_i$ denote the $5\times 5$ matrix from the table above. We display the eigenvalues for $\mathcal{M}_i$. For each eigenvalue, we give a corresponding row eigenvector and column eigenvector. We show that the eigenvalues of $\mathcal{M}_i$ are the same as the eigenvalues of the local graph $\Gamma(x)$.

This paper is organized as follows. In Sections 2 and 3, we present some preliminaries on the Grassmann graph $J_q(n,k)$ and the projective geometry $P_q(n)$. In Section 4, we represent the elements of $P_q(n)$ as vectors in the Euclidean space $E$. In Section 5, we present some results about $\Stab(x,y)$ and $\Fix(x,y)$. In Section 6, we find all the orbits of the $\Stab(x,y)$-action on $\Gamma(x)$. In Sections 7 and 8, we define the vectors $A_{xy}^{+}, A_{xy}^{0}, A_{xy}^{-}$ and write these vectors in terms of the basis (\ref{introbasis1}) and the basis (\ref{introbasis2}). We also obtain inner products between the vectors $\widehat{x},\widehat{y},B_{xy},C_{xy}A_{xy}^{+}, A_{xy}^{0}, A_{xy}^{-}$. In Section 9, we use the matrix $\mathcal{M}_{i}$ to describe the adjacency between the $\Stab(x,y)$-orbits. We find the eigenvalues of $\mathcal{M}_i$ and their corresponding row eigenvectors and column eigenvectors. We also show that the eigenvalues of $\mathcal{M}_i$ are the same as the eigenvalues of $\Gamma(x)$. 

\section{The Grassmann graph $\Gamma$}
Let $\Gamma=(X,\mathcal{E})$ denote a finite undirected graph that is connected, without loops or multiple edges, with vertex set $X$, edge set $\mathcal{E}$, and path-length distance function $\partial$. Two vertices $x,y\in X$ are said to be adjacent whenever they form an edge. The diameter $d$ of $\Gamma$ is defined as $d=\max\{\partial(x,y)\mid x,y\in X\}$. For $x\in X$ and an integer $i\geq 0$, define the set $\Gamma_i(x)=\{y\in X\mid \partial(x,y)=i\}$. We abbreviate $\Gamma(x)=\Gamma_1(x)$. The subgraph induced on $\Gamma(x)$ is called the \emph{local graph} of $x$. 

We say that $\Gamma$ is \textit{regular with valency $\kappa$} whenever $\bigl\vert\Gamma(x)\bigr\vert=\kappa$ for all $x\in X$. 
We say that $\Gamma$ is \textit{distance-regular} whenever for all integers $h,i,j$ such that $0\leq h,i,j\leq d$ and all $x,y\in X$ such that $\partial(x,y)=h$, the cardinality of the set $\{z\in X\mid\partial(x,z)=i, \partial(y,z)=j\}$ depends only on $h,i,j$. This cardinality is denoted by $p_{i,j}^{h}$. For the rest of this section, we assume that $\Gamma$ is
distance-regular with diameter $d\geq 3$. Observe that $\Gamma$ is regular with valency $\kappa=p^{0}_{1,1}$.

Define 
\begin{equation*}
        b_i=p^{i}_{1,i+1}\; \;  (0\leq i<d),\qquad \qquad a_i=p^{i}_{1,i} \; \; (0\leq i\leq d),\qquad \qquad c_i=p^{i}_{1,i-1}\; \; (0<i\leq d).
\end{equation*}
Note that $b_0=\kappa$, $a_0=0$, $c_1=1$. Also note that  
\begin{equation*}
    b_i+a_i+c_i=\kappa \qquad \qquad (0\leq i\leq d),
\end{equation*}
where $c_0=0$ and $b_d=0$.

We call $b_i$, $a_i$, $c_i$ the {\it intersection numbers of $\Gamma$}.

By the {\it eigenvalues of $\Gamma$} we mean the roots of the minimal polynomial of the adjacency matrix. Since $\Gamma$ is distance-regular, by \cite[p.~128]{BCN}, $\Gamma$ has $d+1$ eigenvalues; we denote these eigenvalues by
\begin{equation}
\nonumber
    \theta_0>\theta_1>\cdots>\theta_d.
\end{equation} 

By \cite[p.~129]{BCN}, $\theta_0=\kappa$.

By the {\it spectrum of $\Gamma$} we mean the set of ordered pairs $\bigl\{\left(\theta_{i},m_i\right)\bigr\}_{i=0}^{d}$, where $\left\{\theta_i\right\}_{i=0}^d$ are the
eigenvalues of $\Gamma$ and $m_i$ the dimension of the $\theta_i$-eigenspace ($0\leq i\leq d$).

This paper is about a class of distance-regular graphs called the Grassmann graphs. These graphs are defined as follows. Let $\mathbb{F}=\mathbb{F}_q$ denote a finite field with $q$ elements, and let $n,k$ denote positive integers such that $n> k$. Let $V$ denote an $n$-dimensional vector space over $\mathbb{F}$. The Grassmann graph $J_q(n,k)$ has vertex set $X$ consisting of the $k$-dimensional subspaces of $V$. Vertices $x,y$ of $J_q(n,k)$ are adjacent whenever $x \cap y$ has dimension $k-1$.

According to \cite[p.~268]{BCN}, the graphs $J_q(n,k)$ and $J_q(n,n-k)$ are isomorphic. Without loss of generality, we may assume $n\geq 2k$. Under this assumption, the diameter of $J_q(n,k)$ is equal to $k$. (See \cite[Theorem~9.3.3]{BCN}.) The case $n=2k$ is somewhat special, so throughout this paper we assume that $n>2k$. 

For the rest of this paper, we assume that $\Gamma$ is the Grassmann graph $J_q(n,k)$ with $k\geq 3$.

In what follows, we will use the notation 
\begin{equation*}
    [m]=\frac{q^{m}-1}{q-1} \qquad \qquad(m\in \mathbb{Z}).
\end{equation*} 

By \cite[Theorem~9.3.2]{BCN}, the valency of $\Gamma$ is 
\begin{equation*}
  \kappa=q[k][n-k].  
\end{equation*}

By \cite[Theorem~9.3.3]{BCN}, the intersection numbers of $\Gamma$ are
\begin{equation}
\label{sizebc}
    b_i=q^{2i+1}[k-i][n-k-i],\qquad \qquad c_i=[i]^2 \qquad \qquad (0\leq i\leq k).
\end{equation}

By \cite[Theorem~9.3.3]{BCN}, the eigenvalues of $\Gamma$ are
 \begin{equation}
 \label{eigenvalues}
     \theta_i=q^{i+1}[k-i][n-k-i]-[i] \qquad \qquad (0\leq i\leq k).
 \end{equation}

The given ordering of the eigenvalues is known to be $Q$-polynomial in the sense of \cite[Theorem~8.1.1]{BCN}.

\section{The projective geometry $P_q(n)$}
To study the graph $\Gamma$, it is helpful to view its vertex set $X$ as a subset of a certain poset $P_q(n)$, which is defined as follows. 
\begin{definition}
    Let the poset $P_q(n)$ consist of the subspaces of $V$, together with the partial order given by inclusion. This poset $P_q(n)$ is called the \emph{projective geometry}.
\end{definition}
For the rest of the paper, we abbreviate $P=P_q(n)$. In this section we present some lemmas about the poset $P$. 

\begin{lemma}{\rm{\cite[p.~47]{Axler}}}
    \label{modularity}
    For $u,v\in P$ we have 
    \begin{equation*}
        \dim u+\dim v=\dim \left(u\cap v\right)+\dim\left(u+v\right).
    \end{equation*}
\end{lemma}

\begin{lemma}{{\rm{\cite[Lemma~3.3]{Seong}}}}
    \label{linin}
    Let $u,v\in P$. Let the subset $\mathcal{R}\subseteq V$ form a basis for $u\cap v$. Extend the basis $\mathcal{R}$ to a basis $\mathcal{R}\cup \mathcal{S}$ for $u$, and extend the basis $\mathcal{R}$ to a basis $\mathcal{R}\cup \mathcal{T}$ for $v$. Then $\mathcal{R}\cup \mathcal{S}\cup \mathcal{T}$ forms a basis for the subspace $u+v$.
\end{lemma}

For $0\leq \ell\leq n$, let the set $P_{\ell}$ consist of the $\ell$-dimensional subspaces of $V$. 
Note that $X=P_k$. Also note that $P_0=\{0\}$ and $P_n=\{V\}$.

\begin{lemma}
\label{introlem0}
For $x,y\in X$ the following (i), (ii) hold:
\begin{enumerate}[label={\rm{(\roman*)}}]
\item {\rm{\cite[p.~269]{BCN}}} the dimension of $x\cap y$ is $k-\partial(x,y)$;
\item {\rm{\cite[Lemma~3.5]{Seong}}} the dimension of $x+y$ is $k+\partial(x,y)$.
\end{enumerate}
\end{lemma}

\begin{definition}
\label{defs}
For $u\in P$ define the set 
\begin{equation*}
  \Omega(u)=\{s\in P_1\mid s\subseteq u\}.  
\end{equation*}
\end{definition}

Note that $\Omega(V)=P_1$.

By {\rm{\cite[Section~3]{Seong}}}, the following {\rm{(i)--(ii)}} hold:
    \begin{enumerate}[label={\rm{(\roman*)}}]
        \item for all $u\in P$,
            \begin{equation*}
            \bigl\vert\Omega(u)\bigr\vert=[m],
            \end{equation*}
            where $u\in P_{m}$;
        \item $\vert P_1\vert = [n]$.
    \end{enumerate}

We now comment on the symmetries of $P$. Recall that the general linear group $GL(V)$ consists of the invertible $\mathbb{F}$-linear maps from $V$ to $V$. The action of $GL(V)$ on $V$ induces a permutation action of $GL(V)$ on the set $P$. This permutation action respects the partial order on $P$. The orbits of the action are $P_\ell$ for $0\leq\ell\leq n$. By \cite[Lemma~3.9]{Seong}, the action of $GL(V)$ on $X$ preserves the path-length distance $\partial$.

\section{Representing $P$ using a Euclidean space $E$}

In \cite[Section~4]{Seong} we described how to represent the elements of $P$
as vectors in a Euclidean space. Our goal in this section is to
summarize the description. The material in this section will
be used to state and prove our main results later in the paper.

There are two stages to representing the elements of $P$ as vectors in a Euclidean space. In the first stage we consider the elements of $P_1$.

Let $E$ denote a Euclidean space with dimension $[n]-1$ and bilinear form $\left<\; ,\;\right>$. Recall the notation $\Vert \nu\Vert^2=\left<\nu,\nu\right>$ for any $\nu\in E$. We define a function 
\begin{equation}
    \begin{aligned}
    \label{basichat}
        P_1&\rightarrow E\\s&\mapsto\widehat{s}
    \end{aligned}
\end{equation}
that satisfies the following conditions (C1) $-$ (C4):

\begin{enumerate}[label=(C\arabic*)]
    \item $E=\text{Span}\bigl\{\widehat{s}\mid s\in P_1\bigr\}$;
    
    \item for $s\in P_1$, $\bigl\Vert\widehat{s}\bigr\Vert^2=[n]-1$;
    
    \item for distinct $s,t\in P_1$, $\Bigl<\widehat{s},\widehat{t}\Bigr>=-1$;
    
    \item $\mathlarger{\sum_{s\in P_1}\widehat{s}=0}$.
\end{enumerate}

Next, we extend the function (\ref{basichat}) to a function
\begin{equation}
    \begin{aligned}
    \label{hardhat}
        P&\rightarrow E \\
        u&\mapsto \widehat{u}
    \end{aligned}
\end{equation}
such that for all $u\in P$, 
\begin{equation*}
    \widehat{u}=\sum_{s\in \Omega(u)}\widehat{s}.
\end{equation*}

Note that $\widehat{u}=0$ if $u\in P_0$ or $u\in P_n$. 

Next we present a lemma that involves the map (\ref{hardhat}).

\begin{lemma}
    \label{introlem1}
    The following {\rm{(i)--(vi)}} hold:
    \begin{enumerate}[label={\rm{(\roman*)}}]
        \item {\rm{\cite[Lemma~6.2]{Seong}}} for $u,v\in P$, 
        \begin{equation*}
            \bigl<\widehat{u},\widehat{v}\bigr>=[n][h]-[i][j],
        \end{equation*}
        where
        \begin{equation*}
            i=\dim u,\qquad \qquad j=\dim v, \qquad \qquad h=\dim \left(u\cap v\right);
        \end{equation*}
        
        \item {\rm{\cite[Lemma~6.3]{Seong}}} for $u\in P$,
        \begin{equation*}
                \bigl\Vert\widehat{u}\bigr\Vert^2=q^{i}[i][n-i],
        \end{equation*}
        where $i=\dim u$;

        \item {\rm{\cite[Lemma~6.4]{Seong}}}
        for $x,y\in X$,
        \begin{equation*}
            \bigl<\widehat{x},\widehat{y}\bigr>=[n][k-i]-[k]^2,
        \end{equation*}
        where $i=\partial(x,y)$;

        \item {\rm{\cite[Lemma~6.5]{Seong}}}
        for $x\in X$,
        \begin{equation*}
            \bigl\Vert\widehat{x}\bigr\Vert^2=q^k[k][n-k];
        \end{equation*}

        \item {\rm{\cite[Lemma~6.6]{Seong}}} 
        for $x\in X$,
        \begin{equation*}
            \sum_{z\in \Gamma(x)}\widehat{z}=\theta_1\widehat{x},
        \end{equation*}
        where $\theta_1$ is from {\rm(\ref{eigenvalues})};

        \item {\rm{\cite[Lemma~6.7]{Seong}}} the vector space $E$ is spanned by $\bigl\{\widehat{x}\mid x\in X\bigr\}$.
    \end{enumerate}
\end{lemma}

By \cite[Section~6]{Seong}, the Euclidean space $E$, together with the restriction of the map (\ref{hardhat}) to $X$ gives a Euclidean representation of $\Gamma$ in the sense of \cite[Definition~6.1]{Seong}. This representation is associated with the eigenvalue $\theta_1$. 

By \cite[Lemma~4.2]{Seong}, the Euclidean space $E$ becomes a $GL(V)$-module such that for all $u\in P$ and $\sigma\in GL(V)$, 
\begin{equation*}
    \sigma\bigl(\widehat{u}\bigr)=\widehat{\sigma(u)}.
\end{equation*}
By \cite[Section~6]{Seong}, the Euclidean space $E$ is irreducible as a $GL(V)$-module.

\section{The stabilizer of some elements in $X$}
\label{stabilizer}
In this section, we consider some stabilizer subgroups of $GL(V)$. These subgroups are the stabilizer of a vertex in $X$, and the stabilizer of two distinct vertices in $X$. We obtain some results about these stabilizers that will be used later in the paper.

For $x\in X$, let $\Stab(x)$ denote the subgroup of $GL(V)$ consisting of the elements that fix $x$. We call $\Stab(x)$ the {\it stabilizer of $x$ in $GL(V)$}. 

\begin{lemma}{{\rm{\cite[Lemma~5.1]{Seong}}}}
\label{staborbit}
    For $v,v'\in P$ and $x\in X$, the following are equivalent:
    \begin{enumerate}[label=\rm{{\rm{(\roman*)}}}]
        \item $\dim v=\dim v'$ and $\dim \left(v\cap x\right)=\dim \bigl(v'\cap x\bigr)$;

        \item the subspaces $v$ and $v'$ are contained in the same orbit of the $\Stab(x)$-action on $P$.
    \end{enumerate}
\end{lemma}

Pick distinct $x,y\in X$. Let $\Stab(x,y)$ denote the subgroup of $GL(V)$ consisting of the elements that fix both $x$ and $y$. We call $\Stab(x,y)$ the {\it stabilizer of $x$ and $y$ in $GL(V)$}. 

Let $\Fix(x,y)$ denote the subspace of $E$ consisting of the vectors that are fixed by every element of $\Stab(x,y)$.

\begin{lemma}{{\rm{\cite[Theorem~8.3]{Seong}}}}
    \label{1basis}
    Pick distinct $x,y\in X$. In the table below, we display vectors that form a basis for $\Fix(x,y)$. 
{\rm{
    \begin{center}
        \begin{tabular}{c|c}
            Case & basis for $\Fix(x,y)$  \\
            \hline
             & \\
            $1\leq \partial(x,y)<k$\; & \;$\widehat{x},\qquad \widehat{y}, \qquad \widehat{x\cap y}, \qquad \widehat{x+y}$\\
            & \\
            $\partial(x,y)=k$\; & \;$\widehat{x},\qquad \widehat{y}, \qquad \widehat{x+y}$\\
        \end{tabular}
    \end{center}
}}
\end{lemma}
\; \\
\begin{definition}
\label{geometricdef}
    Pick distinct $x,y\in X$. By the \emph{geometric basis for $\Fix(x,y)$}, we mean the basis displayed in Lemma \ref{1basis}.
\end{definition} 

Note that the case $\partial(x,y)=k$ is special. The case $\partial(x,y)=1$ is also special; see \cite[Definition~9.1,~9.5]{Seong}.

For the rest of the paper, we assume that $1<\partial(x,y)<k$. 

\section{The $y$-partition of $\Gamma(x)$}
\label{ypartition}
Pick $x,y\in X$ such that $1 <\partial(x,y)<k$. In this section we describe the orbits of $\Stab(x,y)$ acting on $\Gamma(x)$.
We will show that there are five orbits. The partition of $\Gamma(x)$ into these five orbits will be called the $y$-partition of $\Gamma(x)$.

\begin{definition}
\label{calbcdef}
For $x,y\in X$ such that $1<\partial(x,y)<k$, define 
\begin{align*}
    \mathcal{B}_{xy}&=\bigl\{z\in \Gamma(x)\mid \partial(y,z)=\partial(x,y)+1\bigr\},\\
    \mathcal{C}_{xy}&=\bigl\{z\in \Gamma(x)\mid \partial(y,z)=\partial(x,y)-1\bigr\},\\
    \mathcal{A}_{xy}&=\bigl\{z\in \Gamma(x)\mid \partial(y,z)=\partial(x,y)\bigr\}.
\end{align*}
\end{definition}
Observe that 
\begin{equation*}
    \bigl\vert\mathcal{B}_{xy}\bigr\vert=b_i, \qquad \qquad \bigl\vert\mathcal{C}_{xy}\bigr\vert=c_i, \qquad \qquad \bigl\vert\mathcal{A}_{xy}\bigr\vert=a_i \qquad \qquad \bigl(i=\partial(x,y)\bigr).
\end{equation*}

\begin{lemma}{{\rm{\cite[Lemma~9.2,~9.4]{Seong}}}}
\label{bcorbit}
     For $x,y\in X$ such that $1<\partial(x,y)<k$, the sets $\mathcal{B}_{xy}$ and $\mathcal{C}_{xy}$ are orbits of the $\Stab(x,y)$-action on $\Gamma(x)$.
\end{lemma}

\begin{definition}
\label{bcdef}
For $x,y\in X$ such that $1<\partial(x,y)<k$, define the vectors
\begin{equation*}
    B_{xy}=\sum_{z\in \mathcal{B}_{xy}}\widehat{z},\qquad \qquad C_{xy}=\sum_{z\in \mathcal{C}_{xy}}\widehat{z},\qquad \qquad  A_{xy}=\sum_{z\in \mathcal{A}_{xy}}\widehat{z}.
\end{equation*}
Note that $B_{xy}, C_{xy}, A_{xy}$ are contained in $E$. We call $B_{xy}, C_{xy}, A_{xy}$ the \emph{characteristic vectors} of $\mathcal{B}_{xy},\mathcal{C}_{xy},\mathcal{A}_{xy}$ respectively.
\end{definition}

\begin{lemma}{{\rm{\cite[Theorem~11.1]{Seong}}}}
\label{combibasis}
    For $x,y\in X$ such that $1<\partial(x,y)<k$, the following vectors form a basis for $\Fix(x,y)$:
    \begin{equation}
\label{basis2}
    \widehat{x},\qquad \qquad  \widehat{y}, \qquad \qquad B_{xy},\qquad \qquad  C_{xy}.
\end{equation}
\end{lemma}

\begin{definition}
    Let $x,y\in X$ satisfy $1<\partial(x,y)<k$. By the \emph{combinatorial basis for $\Fix(x,y)$}, we mean the basis formed by the vectors in (\ref{basis2}).
\end{definition}

Next we focus on the set $\mathcal{A}_{xy}$. This set turns out to be the disjoint union of three orbits of the $\Stab(x,y)$-action on $\Gamma(x)$. Our next general goal is to describe these three orbits.

\begin{definition}
\label{caladef}
For $x,y\in X$ such that $1<\partial(x,y)<k$, define 
\begin{align*}
    \mathcal{A}_{xy}^{+}&=\bigl\{z\in \mathcal{A}_{xy}\mid z+x+y\supsetneq x+y, z\cap x\cap y=x\cap y\bigr\},\\
    \mathcal{A}_{xy}^{0}&=\bigl\{z\in \mathcal{A}_{xy}\mid z+x+y=x+y, z\cap x\cap y=x\cap y\bigr\},\\
    \mathcal{A}_{xy}^{-}&=\bigl\{z\in \mathcal{A}_{xy}\mid z+x+y=x+y, z\cap x\cap y\subsetneq x\cap y\bigr\}.
\end{align*}
\end{definition}

We are going to show that the three sets in Definition \ref{caladef} are orbits of $\Stab(x,y)$. First we have a few remarks.

\begin{lemma}
\label{acomplete}
    For $x,y\in X$ such that $1<\partial(x,y)<k$, the set $\mathcal{A}_{xy}$ is the disjoint union of the sets $\mathcal{A}_{xy}^{+},\mathcal{A}_{xy}^{0},\mathcal{A}_{xy}^{-}$.
\end{lemma}

\begin{proof}
    By linear algebra, the set $\bigl\{z\in \mathcal{A}_{xy}\mid z+x+y\supsetneq x+y, \;z\cap x\cap y\subsetneq x\cap y\bigr\}$ is empty. The result follows.
\end{proof}

\begin{lemma}
\label{acount}
    For $x,y\in X$ such that $1<\partial(x,y)<k$,
    \begin{equation}
    \label{acount2}
    \bigl\vert \mathcal{A}_{xy}^{+}\bigr\vert=q^{i+1}[i][n-k-i],\qquad \qquad 
    \bigl\vert \mathcal{A}_{xy}^{0}\bigr\vert =(q-1)[i]^2,\qquad \qquad
    \bigl\vert \mathcal{A}_{xy}^{-}\bigr\vert=q^{i+1}[i][k-i],
\end{equation}
where $i=\partial(x,y)$.
\end{lemma}
\begin{proof}
    Routine from counting.
\end{proof}

Observe that the values in (\ref{acount2}) depend only on $\partial(x,y)$.

\begin{definition}
    We refer to Lemma \ref{acount}. For $1<i<k$, define
\begin{equation}
\label{acountdef}
    a_i^{+}=\bigl\vert\mathcal{A}_{xy}^{+}\bigr\vert, \qquad \qquad a_i^{0}=\bigl\vert\mathcal{A}_{xy}^{0}\bigr\vert, \qquad \qquad a_i^{-}=\bigl\vert\mathcal{A}_{xy}^{-}\bigr\vert,
\end{equation}
where $i=\partial(x,y)$.
\end{definition}

Note that $a_i^{+}+a_i^{0}+a_i^{-}=a_i$ for $1<i<k$.

Our next goal is to show that 
$\mathcal{A}_{xy}^{0}$ is an orbit of $\Stab(x,y)$. 

\begin{lemma}
    For $x,y\in X$ such that $1<\partial(x,y)<k$, let $z\in \mathcal{A}_{xy}^{0}$. Then
    \begin{equation*}
        x\cap y\subseteq (z+x)\cap y.
    \end{equation*}
    Moreover,
    \begin{equation*}
        \dim \left(x\cap y\right)+1=\dim \,\bigl((z+x)\cap y\bigr).
    \end{equation*}
\end{lemma}
\begin{proof}
    Routine from the definition of $\mathcal{A}_{xy}^{0}$ and linear algebra.
\end{proof}

\begin{lemma}
\label{eta}
    For $x,y\in X$ such that $1<\partial(x,y)<k$, let $z\in \mathcal{A}_{xy}^{0}$. Then there exist vectors
    \begin{equation*}
        \psi\in (z+x)\cap y, \qquad \qquad \eta\in z,\qquad \qquad \varrho\in x
    \end{equation*}
    such that 
    \begin{equation*}
        \quad \psi\not \in x\cap y, \qquad \qquad \eta\not\in z\cap x, \qquad \qquad \varrho\not\in z\cap x,
    \end{equation*}
    \begin{equation*}
        \quad \psi=\eta+\varrho.
    \end{equation*}
\end{lemma}

\begin{proof}
    Pick $\psi\in (z+x)\cap y$ such that $\psi\not\in x\cap y$. Note that $\psi\in z+x$. Also note that $\psi\not\in x$ and $\psi\not\in z$.

    Hence, $\psi$ is a linear combination of some nonzero vector $\eta\in z$ and some nonzero vector $\varrho\in x$. We assume without loss that $\psi=\eta+\varrho$.

    Assume that $\eta\in z\cap x$. Then $\psi=\eta+\varrho\in x$, which is a contradiction. Hence, $\eta\not\in z\cap x$.

    Assume that $\varrho\in z\cap x$. Then $\psi=\eta+\varrho\in z$, which is a contradiction. Hence, $\varrho\not\in z\cap x$. The result follows.
\end{proof}

\begin{lemma}
\label{6cor}
    For $x,y\in X$ such that $1<\partial(x,y)<k$, let $z\in \mathcal{A}_{xy}^{0}$. Let the vectors $\psi,\eta,\varrho$ be from Lemma \ref{eta}.
    Then
    \begin{align}
        &z+\mathbb{F}\psi=z+x, \qquad \qquad z+\mathbb{F}\varrho=z+x, \qquad \qquad (z\cap x)+\mathbb{F}\eta=z, \label{6cor1}\\
        &x+\mathbb{F}\psi=z+x, \qquad \qquad x+\mathbb{F}\eta=z+x, \qquad \qquad (z\cap x)+\mathbb{F}\varrho=x,\label{6cor2}\\
        &\qquad \qquad \qquad \qquad \; (x\cap y)+\mathbb{F}\psi=(z+x)\cap y.\label{6cor3}
    \end{align}
    Moreover, for each equation in {\rm{(\ref{6cor1})}}, {\rm{(\ref{6cor2})}}, {\rm{(\ref{6cor3})}} the sum on the left is direct.
\end{lemma}

\begin{proof}
    Immediate from linear algebra.
\end{proof}

\begin{lemma}
\label{a0orbit}
    For $x,y\in X$ such that $1<\partial(x,y)<k$, the set $\mathcal{A}_{xy}^{0}$ is an orbit of the $\Stab(x,y)$-action on $\Gamma(x)$.
\end{lemma}

\begin{proof}
    By Lemma \ref{staborbit}, the set $\mathcal{A}_{xy}^{0}$ is a disjoint union of orbits of $\Stab(x,y)$. We now show that $\mathcal{A}_{xy}^{0}$ is a single orbit.

    Let $z,z'\in \mathcal{A}_{xy}^{0}$. It suffices to show that there exists $\sigma\in \Stab(x,y)$ that sends $z\mapsto z'$. 

    Let the vectors $\psi, \eta, \varrho$ be from Lemma \ref{eta}.
    Let the subset $\mathcal{R}\subseteq V$ form a basis for $x\cap y$. 
    Extend the basis $\mathcal{R}$ for $x\cap y$ to a basis $\mathcal{R}\cup \mathcal{S}$ for $z\cap x$. 
    By the third equation in (\ref{6cor1}), $\mathcal{R}\cup \mathcal{S}\cup \{\eta\}$ forms a basis for $z$. 
    By the third equation in (\ref{6cor2}), $\mathcal{R}\cup \mathcal{S}\cup \{\varrho\}$ forms a basis for $x$.
    By (\ref{6cor3}), $\mathcal{R}\cup \{\psi\}$ forms a basis for $(z+x)\cap y$. 
    By the first equation in (\ref{6cor2}), $\mathcal{R}\cup \mathcal{S}\cup \{\psi,\varrho\}$ forms a basis for $z+x$. 
    Extend the basis $\mathcal{R}\cup \{\psi\}$ for $(z+x)\cap y$ to a basis $\mathcal{R}\cup \mathcal{Q}\cup \{\psi\}$ for $y$.
    By Lemma \ref{linin}, $\mathcal{R}\cup \mathcal{S}\cup \mathcal{Q}\cup \{\psi,\varrho\}$ forms a basis for $x+y$. 
    Extend the basis $\mathcal{R}\cup \mathcal{S}\cup \mathcal{Q}\cup \{\psi,\varrho\}$ for $x+y$ to a basis $\mathcal{R}\cup \mathcal{S}\cup \mathcal{Q}\cup \mathcal{W}\cup \{\psi,\varrho\}$ for $V$.

    Recall the element $z'\in \mathcal{A}_{xy}^{0}$. Consider the corresponding vectors $\psi', \eta', \varrho'$ from Lemma \ref{eta}.
    Extend the basis $\mathcal{R}$ for $x\cap y$ to a basis $\mathcal{R}\cup \mathcal{S}'$ for $z'\cap x$. 
    By the third equation in (\ref{6cor1}), $\mathcal{R}\cup \mathcal{S}'\cup \bigl\{\eta'\bigr\}$ forms a basis for $z'$. 
    By the third equation in (\ref{6cor2}), $\mathcal{R}\cup \mathcal{S}'\cup \bigl\{\varrho'\bigr\}$ forms a basis for $x$.
    By (\ref{6cor3}), $\mathcal{R}\cup \bigl\{\psi'\bigr\}$ forms a basis for $\bigl(z'+x\bigr)\cap y$. 
    By the first equation in (\ref{6cor2}), $\mathcal{R}\cup \mathcal{S}'\cup \bigl\{\psi',\varrho'\bigr\}$ forms a basis for $z'+x$. 
    Extend the basis $\mathcal{R}\cup \bigl\{\psi'\bigr\}$ for $\bigl(z'+x\bigr)\cap y$ to a basis $\mathcal{R}\cup \mathcal{Q}'\cup \bigl\{\psi'\bigr\}$ for $y$.
    By Lemma \ref{linin}, $\mathcal{R}\cup \mathcal{S}'\cup \mathcal{Q}'\cup \bigl\{\psi',\varrho'\bigr\}$ forms a basis for $x+y$. 
    Extend the basis $\mathcal{R}\cup \mathcal{S}'\cup \mathcal{Q}'\cup \bigl\{\psi',\varrho'\bigr\}$ for $x+y$ to a basis $\mathcal{R}\cup \mathcal{S}'\cup \mathcal{Q}'\cup \mathcal{W}'\cup \bigl\{\psi',\varrho'\bigr\}$ for $V$.

    By linear algebra, there exists $\sigma\in GL(V)$ that sends $\mathcal{S}\mapsto \mathcal{S}'$, $\mathcal{Q}\mapsto \mathcal{Q}'$, $\mathcal{W}\mapsto \mathcal{W}'$, $\psi\mapsto \psi'$, $\varrho\mapsto \varrho'$ and acts as the identity on $\mathcal{R}$. By construction, $\sigma$ is contained in $\Stab(x,y)$ and sends $z\mapsto z'$. The result follows.
\end{proof}

\begin{lemma}
    \label{apmorbit}
    For $x,y\in X$ such that $1<\partial(x,y)<k$, the sets $\mathcal{A}_{xy}^{+}, \mathcal{A}_{xy}^{-}$ are orbits of the $\Stab(x,y)$-action on $\Gamma(x)$.
\end{lemma}

\begin{proof}
    Similar to Lemma \ref{a0orbit}.
\end{proof}

\begin{theorem}
\label{5orbitstheorem}
    For $x,y\in X$ such that $1<\partial(x,y)<k$, the following sets are orbits of the $\Stab(x,y)$-action on $\Gamma(x)$:
    \begin{equation}
        \label{5orbits}
        \mathcal{B}_{xy}, \qquad \qquad \mathcal{C}_{xy}, \qquad \qquad \mathcal{A}_{xy}^{+}, \qquad \qquad \mathcal{A}_{xy}^{0}, \qquad \qquad \mathcal{A}_{xy}^{-}.
    \end{equation}
    Furthermore, these orbits form a partition of $\Gamma(x)$.
\end{theorem}

\begin{proof}
    For the first assertion, combine Lemmas \ref{bcorbit}, \ref{a0orbit}, \ref{apmorbit}. The second assertion is immediate from Lemma \ref{acomplete} and the fact that the disjoint union of $\mathcal{B}_{xy}, \mathcal{C}_{xy}, \mathcal{A}_{xy}$ is equal to $\Gamma(x)$.
\end{proof}

\begin{definition}
    Let $x,y\in X$ satisfy $1<\partial(x,y)<k$, and consider the partition of $\Gamma(x)$ given in (\ref{5orbits}). By construction, this partition is equitable in the sense of \cite[p.~159]{Schwenk}. We call this partition the {\it $y$-partition of $\Gamma(x)$}.
\end{definition}

\section{The vectors $A_{xy}^{+},A_{xy}^{0},A_{xy}^{-}$}

Pick $x,y\in X$ such that $1<\partial(x,y)<k$. Recall the sets $\mathcal{A}_{xy}^{+}, \mathcal{A}_{xy}^{0}, \mathcal{A}_{xy}^{-}$ from Definition \ref{caladef}. In this section we use these sets to define some vectors $A_{xy}^{+},A_{xy}^{0},A_{xy}^{-}$ in the Euclidean space $E$. We show that $A_{xy}^{+},A_{xy}^{0},A_{xy}^{-}$ are contained in $\Fix(x,y)$. We write $A_{xy}^{+},A_{xy}^{0},A_{xy}^{-}$ in terms of the geometric basis for $\Fix(x,y)$ and also the combinatorial basis for $\Fix(x,y)$.

\begin{definition}
\label{adef}
For $x,y\in X$ such that $1<\partial(x,y)<k$, define the vectors
\begin{equation}
\label{adef2}
   A_{xy}^{+}=\sum_{z\in \mathcal{A}_{xy}^{+}}\widehat{z},\qquad \qquad A_{xy}^{0}=\sum_{z\in \mathcal{A}_{xy}^{0}}\widehat{z},\qquad \qquad A_{xy}^{-}=\sum_{z\in \mathcal{A}_{xy}^{-}}\widehat{z}.
\end{equation}
Note that $A_{xy}^{+}, A_{xy}^{0},A_{xy}^{-}$ are contained in $E$. We call $A_{xy}^{+}, A_{xy}^{0}, A_{xy}^{-}$ the \emph{characteristic vectors} of $\mathcal{A}_{xy}^{+},\mathcal{A}_{xy}^{0},\mathcal{A}_{xy}^{-}$ respectively.
\end{definition}

By Lemma \ref{acomplete}, $A_{xy}= A_{xy}^{+}+ A_{xy}^{0}+ A_{xy}^{-}$.

\begin{lemma}
    For $x,y\in X$ such that $1<\partial(x,y)<k$, the vectors $A_{xy}^{+}, A_{xy}^{0}, A_{xy}^{-}$ are contained in $\Fix(x,y)$.
\end{lemma}

\begin{proof}
    Pick $\sigma\in \Stab(x,y)$. Since $\mathcal{A}_{xy}^{+},\mathcal{A}_{xy}^{0},\mathcal{A}_{xy}^{-}$ are orbits of the $\Stab(x,y)$-action on $\Gamma(x)$, the map $\sigma$ fixes $\mathcal{A}_{xy}^{+},\mathcal{A}_{xy}^{0},\mathcal{A}_{xy}^{-}$. The result follows. 
\end{proof}

Our next goal is to write $A_{xy}^{+},A_{xy}^{0},A_{xy}^{-}$ in terms of the geometric basis for $\Fix(x,y)$. To do this, we recall the inner products that involve the vectors in the geometric basis for $\Fix(x,y)$.

\begin{lemma}{{\rm{\cite[Theorem~10.4]{Seong}}}}
\label{innprodmat21}
Let $x,y\in X$ satisfy $1<\partial(x,y)<k$. In the following table, for each vector $u$ in the header column, and each vector $v$ in the header row, the $(u,v)$-entry of the table gives the inner product $\left<u,v\right>$. Write $i=\partial(x,y)$.

\begin{center}
    \begin{tabular}{c|c c c c}
        $\left<\;,\;\right>$ & $\widehat{x}$ & $\widehat{y}$ & $\widehat{x\cap y}$ & $\widehat{x+y}$\\ 
         \hline
         \\
         $\widehat{x}$ & $q^k[k][n-k]$ & $[n][k-i]-[k]^2$ & $q^k[k-i][n-k]$ & $q^{k+i}[k][n-k-i]$\\
         \\
        $\widehat{y}$ & $[n][k-i]-[k]^2$ & $q^{k}[k][n-k]$ & $q^k[k-i][n-k]$ & $q^{k+i}[k][n-k-i]$\\
        \\
        $\widehat{x\cap y}$&$q^k[k-i][n-k]$ & $q^k[k-i][n-k]$ & $q^{k-i}[k-i][n-k+i]$ & $q^{k+i}[k-i][n-k-i]$\\
        \\
        $\widehat{x+y}$&$q^{k+i}[k][n-k-i]$ & $q^{k+i}[k][n-k-i]$ & $q^{k+i}[k-i][n-k-i]$ & $q^{k+i}[k+i][n-k-i]$\\
    \end{tabular}
\end{center}
\end{lemma}
\; \\
For $1<i<k$ let $M_i$ denote the matrix of inner products in Lemma \ref{innprodmat21}.

\begin{lemma}{{\rm{\cite[Lemma~10.10]{Seong}}}}
\label{inverse}
For $1<i<k$ the inverse of the matrix $M_i$ is given by 
\begin{equation*}
    M_{i}^{-1}=\frac{1}{q^{k-i}(q-1)[i]^2[n]}\begin{pmatrix}
q^i&1&-q^i&-1\\
1 & q^i & -q^i & -1\\
-q^i & -q^i & \frac{q^i[k]-[i]}{[k-i]} & 1\\
-1 & -1 & 1 & \frac{q^i[n-k]-[i]}{q^{2i}[n-k-i]}
\end{pmatrix}.
\end{equation*}
\end{lemma}

Next we find inner products that involve the vectors $A_{xy}^{+},A_{xy}^{0},A_{xy}^{-}$.

\begin{lemma}
\label{ax}
    For $x,y\in X$ such that $1<\partial(x,y)<k$, we have
    \begin{align}
        \Bigl<A_{xy}^{+},\widehat{x}\Bigr>&= q^{i+1}[i][n-k-i]\Bigl([n][k-1]-[k]^2\Bigr),\label{apx}\\
        \Bigl<A_{xy}^{0},\widehat{x}\Bigr>&= (q-1)[i]^2\Bigl([n][k-1]-[k]^2\Bigr),\label{a0x}\\
        \Bigl<A_{xy}^{-},\widehat{x}\Bigr>&= q^{i+1}[i][k-i]\Bigl([n][k-1]-[k]^2\Bigr),\label{amx}
    \end{align}
    where $i=\partial(x,y)$.
\end{lemma}

\begin{proof}
    We first prove (\ref{apx}). Using the first equation in (\ref{adef2}), we obtain
    \begin{equation}
    \label{apxyx}
        \Bigl<A_{xy}^{+},\widehat{x}\Bigr>=\sum_{z\in \mathcal{A}_{xy}^{+}} \bigl<\widehat{z},\widehat{x}\bigr>.
    \end{equation}

    Pick $z\in \mathcal{A}_{xy}^{+}$. By the definition of $\mathcal{A}_{xy}^{+}$ and Lemma \ref{introlem1}(iii),
    \begin{equation}
    \label{zxinner}
    \bigl<\widehat{z},\widehat{x}\bigr>=[n][k-1]-[k]^2.
    \end{equation}

    By the above comments, 
    \begin{equation}
    \label{apxyx2}
        \Bigl< A_{xy}^{+}, \widehat{x} \Bigr> = \bigl\vert\mathcal{A}_{xy}^{+}\bigr\vert \Bigl( [n][k-1] - [k]^2\Bigr).
    \end{equation}
    In (\ref{apxyx2}), we evaluate $\bigl\vert \mathcal{A}_{xy}^{+}\bigr\vert$ using (\ref{acount2}); this yields (\ref{apx}).

    We have now verified (\ref{apx}). Equations (\ref{a0x}) and (\ref{amx}) are obtained in a similar fashion.
\end{proof}

\begin{lemma}
\label{ay}
    For $x,y\in X$ such that $1<\partial(x,y)<k$, we have
    \begin{align}
        \Bigl<A_{xy}^{+},\widehat{y}\Bigr>&= q^{i+1}[i][n-k-i]\Bigl([n][k-i]-[k]^2\Bigr),\label{apy}\\
        \Bigl<A_{xy}^{0},\widehat{y}\Bigr>&= (q-1)[i]^2\Bigl([n][k-i]-[k]^2\Bigr),\label{a0y}\\
        \Bigl<A_{xy}^{-},\widehat{y}\Bigr>&= q^{i+1}[i][k-i]\Bigl([n][k-i]-[k]^2\Bigr),\label{amy}
    \end{align}
    where $i=\partial(x,y)$.
\end{lemma}

\begin{proof}
    We first prove (\ref{apy}). Using the first equation in (\ref{adef2}), we obtain
    \begin{equation}
    \label{apxyy}
        \Bigl<A_{xy}^{+},\widehat{y}\Bigr>=\sum_{z\in \mathcal{A}_{xy}^{+}} \bigl<\widehat{z},\widehat{y}\bigr>.
    \end{equation}

    Pick $z\in \mathcal{A}_{xy}^{+}$. By the definition of $\mathcal{A}_{xy}^{+}$ and Lemma \ref{introlem1}(iii),
    \begin{equation}
    \label{zyinner}
    \bigl<\widehat{z},\widehat{y}\bigr>=[n][k-i]-[k]^2.
    \end{equation}

    By the above comments, 
    \begin{equation}
    \label{apxyy2}
        \Bigl<A_{xy}^{+}, \widehat{y} \Bigr> = \bigl\vert \mathcal{A}_{xy}^{+}\bigr\vert \Bigl( [n][k-i] - [k]^2\Bigr).
    \end{equation}
    In (\ref{apxyy2}), we evaluate $\bigl\vert \mathcal{A}_{xy}^{+}\bigr\vert$ using (\ref{acount2}); this yields (\ref{apy}).
    
    We have now verified (\ref{apy}). Equations (\ref{a0y}) and (\ref{amy}) are obtained in a similar fashion.
\end{proof}

\begin{lemma}
\label{axny}
    For $x,y\in X$ such that $1<\partial(x,y)<k$, we have
    \begin{align}
        \Bigl<A_{xy}^{+},\widehat{x\cap y}\Bigr>&= q^{k+i+1}[i][n-k-i][k-i][n-k],\label{apxny}\\
        \Bigl<A_{xy}^{0},\widehat{x\cap y}\Bigr>&= q^k(q-1)[i]^2[k-i][n-k],\label{a0xny}\\
        \Bigl<A_{xy}^{-},\widehat{x\cap y}\Bigr>&= q^{i+1}[i][k-i]\Bigl([n][k-i-1]-[k-i][k]\Bigr),\label{amxny}
    \end{align}
    where $i=\partial(x,y)$.
\end{lemma}

\begin{proof}
    We first prove (\ref{apxny}). Using the first equation in (\ref{adef2}), we obtain
    \begin{equation}
    \label{apxyn}
        \Bigl<A_{xy}^{+},\widehat{x\cap y}\Bigr>=\sum_{z\in \mathcal{A}_{xy}^{+}} \Bigl<\widehat{z},\widehat{x\cap y}\Bigr>.
    \end{equation}

    Pick $z\in \mathcal{A}_{xy}^{+}$. By the definition of $\mathcal{A}_{xy}^{+}$ and Lemma \ref{introlem1}(i),
    \begin{equation}
    \label{z+xnyinner}
    \Bigl<\widehat{z},\widehat{x\cap y}\Bigr>=[n][k-i]-[k][k-i]=q^{k}[k-i][n-k].
    \end{equation}

    By the above comments, 
    \begin{equation}
    \label{apxyn2}
        \Bigl<A_{xy}^{+}, \widehat{x\cap y} \Bigr> = \bigl\vert \mathcal{A}_{xy}^{+}\bigr\vert\; q^k[k-i][n-k].
    \end{equation}
    In (\ref{apxyn2}), we evaluate $\bigl\vert \mathcal{A}_{xy}^{+}\bigr\vert$ using (\ref{acount2}); this yields (\ref{apxny}).

    We have now verified (\ref{apxny}). Equation (\ref{a0xny}) is obtained in a similar fashion.

    Next we prove (\ref{amxny}). Using the last equation in (\ref{adef2}), we obtain
    \begin{equation}
    \label{amxyn}
        \Bigl<A_{xy}^{-},\widehat{x\cap y}\Bigr>=\sum_{z\in \mathcal{A}_{xy}^{-}} \Bigl<\widehat{z},\widehat{x\cap y}\Bigr>.
    \end{equation}

    Pick $z\in \mathcal{A}_{xy}^{-}$. By the definition of $\mathcal{A}_{xy}^{-}$ and Lemma \ref{modularity},
    \begin{equation}
    \label{zxydim}
        \dim \left(z\cap x\cap y\right)=k-i-1.
    \end{equation}
    
    By (\ref{zxydim}) and Lemma \ref{introlem1}(i),
    \begin{equation}
    \label{z-xnyinner}
    \Bigl<\widehat{z},\widehat{x\cap y}\Bigr>=[n][k-i-1]-[k-i][k].
    \end{equation}

    By the above comments, 
    \begin{equation}
    \label{amxyn2}
        \Bigl< A_{xy}^{-}, \widehat{x\cap y} \Bigr> = \bigl\vert \mathcal{A}_{xy}^{-}\bigr\vert \Bigl( [n][k-i-1]-[k-i][k]\Bigr).
    \end{equation}
    In (\ref{amxyn2}), we evaluate $\left\vert \mathcal{A}_{xy}^{-}\right\vert$ using (\ref{acount2}); this yields (\ref{amxny}).
\end{proof}

\begin{lemma}
\label{ax+y}
    For $x,y\in X$ such that $1<\partial(x,y)<k$, we have
    \begin{align}
        \Bigl<A_{xy}^{+},\widehat{x+y}\Bigr>&= q^{i+1}[i][n-k-i]\Bigl([n][k-1]-[k][k+i]\Bigr),\label{apx+y}\\
        \Bigl<A_{xy}^{0},\widehat{x+y}\Bigr>&= q^{k+i}(q-1)[i]^2[k][n-k-i],\label{a0x+y}\\
        \Bigl<A_{xy}^{-},\widehat{x+y}\Bigr>&= q^{k+2i+1}[i][k-i][k][n-k-i],\label{amx+y}
    \end{align}
    where $i=\partial(x,y)$.
\end{lemma}

\begin{proof}
We first prove (\ref{apx+y}). Using the first equation in (\ref{adef2}), we obtain
    \begin{equation}
    \label{apxyt}
        \Bigl<A_{xy}^{+},\widehat{x+y}\Bigr>=\sum_{z\in \mathcal{A}_{xy}^{+}} \Bigl<\widehat{z},\widehat{x+y}\Bigr>.
    \end{equation}

    Pick $z\in \mathcal{A}_{xy}^{+}$. By the definition of $\mathcal{A}_{xy}^{+}$ and Lemma \ref{modularity},
    \begin{equation}
    \label{zx+ydim}
        \dim \, \bigl(z\cap (x+y)\bigr)=k-1.
    \end{equation}
    
    By (\ref{zx+ydim}) and Lemma \ref{introlem1}(i),
    \begin{equation}
    \label{z+x+yinner}
    \Bigl<\widehat{z},\widehat{x+y}\Bigr>=[n][k-1]-[k][k+i].
    \end{equation}

    By the above comments, 
    \begin{equation}
    \label{apxyt2}
        \Bigl< A_{xy}^{+}, \widehat{x+y} \Bigr> = \bigl\vert \mathcal{A}_{xy}^{+}\bigr\vert \Bigl([n][k-1]-[k][k+i]\Bigr).
    \end{equation}
    In (\ref{apxyt2}), we evaluate $\bigl\vert \mathcal{A}_{xy}^{+}\bigr\vert$ using (\ref{acount2}); this yields (\ref{apx+y}).  

    Next we prove (\ref{a0x+y}). Using the second equation in (\ref{adef2}), we obtain
    \begin{equation}
    \label{a0xyt}
        \Bigl<A_{xy}^{0},\widehat{x+y}\Bigr>= \sum_{z\in \mathcal{A}_{xy}^{0}} \Bigl<\widehat{z},\widehat{x+y}\Bigr>.
    \end{equation}

    Pick $z\in \mathcal{A}_{xy}^{0}$. By the definition of $\mathcal{A}_{xy}^{0}$, 
    \begin{equation}
    \label{z0x+ydim}
        z\cap (x+y)=z.
    \end{equation}
    
    By (\ref{z0x+ydim}) and Lemma \ref{introlem1}(i),
    \begin{equation}
    \label{z0x+yinner}
    \Bigl<\widehat{z},\widehat{x\cap y}\Bigr>=[n][k]-[k][k+i]=q^{k+i}[k][n-k-i].
    \end{equation}

    By the above comments, 
    \begin{equation}
    \label{a0xyt2}
        \Bigl< A_{xy}^{0}, \widehat{x+y} \Bigr> = \bigl\vert \mathcal{A}_{xy}^{0}\bigr\vert\; q^{k+i}[k][n-k-i].
    \end{equation}
    In (\ref{a0xyt2}), we evaluate $\bigl\vert \mathcal{A}_{xy}^{0}\bigr\vert$ using (\ref{acount2}); this yields (\ref{a0x+y}).

    We have now verified (\ref{a0x+y}). Equation (\ref{amx+y}) is obtained in a similar fashion.
\end{proof}

\begin{theorem}
\label{aprodmat21}
Let $x,y\in X$ satisfy $1<\partial(x,y)<k$. In the following table, for each vector $u$ in the header column, and each vector $v$ in the header row, the $(u,v)$-entry of the table gives the inner product $\left<u,v\right>$. Write $i=\partial(x,y)$.

\begin{center}
    \begin{tabular}{c|c c c}
        $\left<\;,\;\right>$ & $A_{xy}^{+}$ & $A_{xy}^{0}$ & $A_{xy}^{-}$\\ 
         \hline
         \\
         $\widehat{x}$ & $\substack{q^{i+1}[i][n-k-i]\left([n][k-1]-[k]^2\right)}$ & $\substack{(q-1)[i]^2\left([n][k-1]-[k]^2\right)}$ & $\substack{q^{i+1}[i][k-i]\left([n][k-1]-[k]^2\right)}$\\
         \\
         $\widehat{y}$ & $\substack{q^{i+1}[i][n-k-i]\left([n][k-i]-[k]^2\right)}$ & $\substack{(q-1)[i]^2\left([n][k-i]-[k]^2\right)}$ & $\substack{q^{i+1}[i][k-i]\left([n][k-i]-[k]^2\right)}$\\
         \\
         $\widehat{x\cap y}$ & $\substack{q^{k+i+1}[i][n-k-i][k-i][n-k]}$ & $\substack{q^k(q-1)[i]^2[k-i][n-k]}$ & $\substack{q^{i+1}[i][k-i]\left([n][k-i-1]-[k-i][k]\right)}$\\
         \\
         $\widehat{x+y}$ & $\substack{q^{i+1}[i][n-k-i]\left([n][k-1]-[k][k+i]\right)}$ & $\substack{q^{k+i}(q-1)[i]^2[k][n-k-i]}$ & $\substack{q^{k+2i+1}[i][k-i][k][n-k-i]}$\\
    \end{tabular}
\end{center}
\end{theorem}

\begin{proof}
    Combine Lemmas \ref{ax}--\ref{ax+y}.
\end{proof}

In the next result, we write $A_{xy}^{+},A_{xy}^{0},A_{xy}^{-}$ in terms of the geometric basis for $\Fix(x,y)$.
\begin{theorem}
\label{alin}
For $x,y\in X$ such that $1<\partial(x,y)<k$, we have
    \begin{align}
    A_{xy}^{+}&=q^{i+1}[n-k-i][i-1]\widehat{x}+q^{2i}[n-k-i]\widehat{x\cap y}-[i]\widehat{x+y},\label{aplin}\\
    A_{xy}^{0}&=\bigl(q^i[i-1]-[i]\bigr)\widehat{x}-q^{i-1}\widehat{y}+q^{2i-1}\widehat{x\cap y}+q^{i-1}\widehat{x+y},\label{a0lin}\\
    A_{xy}^{-}&=q^{i+1}[k-i][i-1]\widehat{x}-q^i[i]\widehat{x\cap y}+q^{i}[k-i]\widehat{x+y},\label{amlin}
    \end{align}
where $i=\partial(x,y)$.
\end{theorem}

\begin{proof}
    Write 
    \begin{align}
    A_{xy}^{+}&=\alpha\widehat{x}+\beta\widehat{y}+\gamma\widehat{x\cap y}+\delta\widehat{x+y},\label{aplinwrite}\\
    A_{xy}^{0}&=\alpha'\widehat{x}+\beta'\widehat{y}+\gamma'\widehat{x\cap y}+\delta'\widehat{x+y},\label{a0linwrite}\\
    A_{xy}^{-}&=\alpha''\widehat{x}+\beta''\widehat{y}+\gamma''\widehat{x\cap y}+\delta''\widehat{x+y},\label{amlinwrite}
    \end{align}
for $\alpha,\beta,\gamma,\delta,\alpha',\beta',\gamma',\delta',\alpha'',\beta'',\gamma'',\delta''\in \mathbb{R}$. 

Let $N_i$ denote the matrix of inner products from Theorem \ref{aprodmat21}. 

In each of (\ref{aplinwrite}), (\ref{a0linwrite}), (\ref{amlinwrite}) we take the inner product of either side with each of $\widehat{x}, \widehat{y}, \widehat{x\cap y}$, $\widehat{x+y}$ to obtain 
\begin{equation}
        \nonumber
        M_i\begin{pmatrix}
\alpha&\alpha'&\alpha''\\ \beta&\beta'&\beta''\\ \gamma&\gamma'&\gamma''\\ \delta&\delta'&\delta''
\end{pmatrix}=N_i.
\end{equation}

The matrix $M_i$ is invertible by Lemma \ref{inverse}, so

\begin{equation*}
        \begin{pmatrix}
\alpha&\alpha'&\alpha''\\ \beta&\beta'&\beta''\\ \gamma&\gamma'&\gamma''\\ \delta&\delta'&\delta''
\end{pmatrix}
=M_i^{-1}N_i.
\end{equation*}
Using Lemma \ref{inverse} and matrix multiplication we obtain
\begin{equation*}
    M_{i}^{-1}N_i=
    \begin{pmatrix}
q^{i+1}[n-k-i][i-1]& q^i[i-1]-[i] & q^{i+1}[k-i][i-1]\\
0& -q^{i-1} & 0\\
q^{2i}[n-k-i]& q^{2i-1} & -q^i [i]\\
-[i]& q^{i-1} & q^i[k-i]\\
\end{pmatrix}.
\end{equation*}
The result follows.
\end{proof}

Fix $x,y\in X$ such that $1<\partial(x,y)<k$. Our next goal is to write $A_{xy}^{+},A_{xy}^{0},A_{xy}^{-}$ in terms of the combinatorial basis for $\Fix(x,y)$. To do this, we write $\widehat{x\cap y}, \widehat{x+y}$ in terms of the combinatorial basis for $\Fix(x,y)$.

\begin{lemma}{{\rm{\cite[Theorem~11.4]{Seong}}}}
\label{maintheorem}
    For $x,y\in X$ such that $1<\partial(x,y)<k$, we have
    \begin{equation*}
        \begin{aligned}
            \widehat{x\cap y}&=\frac{[k-i][n-k-1]}{q^{k-1}[n-2k]}\widehat{x}+\frac{[k-i]}{q^{k-i+1}[i-1][n-2k]}\widehat{y}-\frac{1}{q^{k+i}[n-2k]}B_{xy}-\frac{[k-i]}{q^k[i-1][n-2k]}C_{xy},\\
            \widehat{x+y}&=-\frac{[k-1][n-k-i]}{q^{k-i-1}[n-2k]}\widehat{x}-\frac{[n-k-i]}{q^{k-2i+1}[i-1][n-2k]}\widehat{y}+\frac{1}{q^k[n-2k]}B_{xy}+\frac{[n-k-i]}{q^{k-i}[i-1][n-2k]}C_{xy},
        \end{aligned}
    \end{equation*}
    where $i=\partial(x,y)$.
\end{lemma}

\begin{theorem}
\label{alin2}
For $x,y\in X$ such that $1<\partial(x,y)<k$, we have
\begin{equation}
    \begin{split}
    \label{aplin2}
        A_{xy}^{+}=\;&\frac{[k-1][n-k-i][n-k]}{q^{k-i-1}[n-2k]}\widehat{x}+\frac{[k][n-k-i]}{q^{k-2i+1}[i-1][n-2k]}\widehat{y}\\
    &\qquad \qquad-\frac{[n-k]}{q^{k}[n-2k]}B_{xy}-\frac{[k][n-k-i]}{q^{k-i}[i-1][n-2k]}C_{xy},\\
    \end{split}
    \end{equation}
    \begin{equation}
    \label{a0lin2}
    A_{xy}^{0}=-
    [i]\widehat{x}-\frac{q^{i-1}[i]}{[i-1]}\widehat{y}+\frac{q^{i-1}}{[i-1]}C_{xy}, \qquad \qquad \qquad \qquad \quad \, \ \ \\
    \end{equation}
    \begin{equation}
    \label{amlin2}
    \begin{split}
        A_{xy}^{-}=\;&-\frac{[k-i][k][n-k-1]}{q^{k-i-1}[n-2k]}\widehat{x}-\frac{[k-i][n-k]}{q^{k-2i+1}[i-1][n-2k]}\widehat{y}\\
    &\qquad \qquad+\frac{[k]}{q^k [n-2k]}B_{xy}+\frac{[k-i][n-k]}{q^{k-i}[i-1][n-2k]}C_{xy},
    \end{split}
\end{equation}
where $i=\partial(x,y)$.
\end{theorem}

\begin{proof}
    We first prove (\ref{aplin2}). In the equation (\ref{aplin}), eliminate $\widehat{x\cap y}$ and $\widehat{x+y}$ using Lemma \ref{maintheorem} and simplify the result.

    We have now verified (\ref{aplin2}). Equations (\ref{a0lin2}) and (\ref{amlin2}) are obtained in a similar fashion.
\end{proof}

\section{Some inner products involving the $y$-partition of $\Gamma(x)$}
Pick $x,y\in X$ such that $1<\partial(x,y)<k$. In this section we calculate the inner products between the vectors $\widehat{x}, \widehat{y}, B_{xy},C_{xy},A_{xy}^{+},A_{xy}^{0},A_{xy}^{-}$. We begin by recalling some inner products from \cite{Seong}.

\begin{lemma}{{\rm{\cite[Theorem~10.9]{Seong}}}}
\label{innprodmat23}
Let $x,y\in X$ satisfy $1<\partial(x,y)<k$. In the following table, for each vector $u$ in the header column, and each vector $v$ in the header row, the $(u,v)$-entry of the table gives the inner product $\left<u,v\right>$. Write $i=\partial(x,y)$.

\begin{center}
    \begin{tabular}{c|c c c c}
        $\left<\;,\;\right>$ & $\widehat{x}$ & $\widehat{y}$ & $\widehat{x\cap y}$ & $\widehat{x+y}$\\ 
         \hline
         \\
        $\widehat{x}$ & $q^k[k][n-k]$ & $[n][k-i]-[k]^2$ & $q^k[k-i][n-k]$ & $q^{k+i}[k][n-k-i]$ \\
        \\
        $\widehat{y}$ & $[n][k-i]-[k]^2$ & $q^k[k][n-k]$ & $q^k[k-i][n-k]$ & $q^{k+i}[k][n-k-i]$ \\
        \\
        $B_{xy}$& $\substack{q^{2i+1}[k-i][n-k-i]\cdot\\ \left([n][k-1]-[k]^2\right)}$ & $\substack{q^{2i+1}[k-i][n-k-i]\cdot\\ \left([n][k-i-1]-[k]^2\right)}$ & $\substack{q^{2i+1}[k-i][n-k-i]\cdot\\ \left([n][k-i-1]-[k-i][k]\right)}$ & $\substack{q^{2i+1}[k-i][n-k-i]\cdot\\ \left([n][k-1]-[k][k+i]\right)}$ \\
        \\
        $C_{xy}$& $[i]^2\bigl([n][k-1]-[k]^2\bigr)$& $[i]^2\bigl([n][k-i+1]-[k]^2\bigr)$ & $q^k[i]^2[k-i][n-k]$ & $q^{k+i}[i]^2[k][n-k-i]$\\
    \end{tabular}
\end{center}
\end{lemma}
\; \\
\begin{lemma}{{\rm{\cite[Theorem~10.15]{Seong}}}}
\label{innprodmat22}
Let $x,y\in X$ satisfy $1<\partial(x,y)<k$. In the following table, for each vector $u$ in the header column, and each vector $v$ in the header row, the $(u,v)$-entry of the table gives the inner product $\left<u,v\right>$. Write $i=\partial(x,y)$.
\begin{center}
    \begin{tabular}{c|c c c c}
        $\left<\;,\;\right>$ & $\widehat{x}$ & $\widehat{y}$ & $B_{xy}$ & $C_{xy}$\\ 
         \hline
         \\
         $\widehat{x}$ & $q^k[k][n-k]$ & $[n][k-i]-[k]^2$ & $\substack{q^{2i+1}[k-i][n-k-i]\cdot\\ \left([n][k-1]-[k]^2\right)}$ & $\substack{[i]^2\left([n][k-1]-[k]^2\right)}$\\
         \\
        $\widehat{y}$ & $[n][k-i]-[k]^2$ & $q^{k}[k][n-k]$ & $\substack{q^{2i+1}[k-i][n-k-i]\cdot\\ \left([n][k-i-1]-[k]^2\right)}$ & $\substack{[i]^2\left([n][k-i+1]-[k]^2\right)}$\\
        \\
        $B_{xy}$&$\substack{q^{2i+1}[k-i][n-k-i]\cdot\\\left([n][k-1]-[k]^2\right)}$ & $\substack{q^{2i+1}[k-i][n-k-i]\cdot\\\left([n][k-i-1]-[k]^2\right)}$ & $\substack{q^{4i+2}[k-i][n-k-i]\cdot\\ \bigl(q^{k-i-2}[n]\left([k-i]+[n-k-i]\right)+\\
    [k-i][n-k-i]\left([n][k-2]-[k]^2\right)\bigr)}$ & $\substack{q^{2i+1}[k-i][n-k-i]\cdot\\ [i]^2\left([n][k-2]-[k]^2\right)}$\\
        \\
        $C_{xy}$&$\substack{[i]^2\left([n][k-1]-[k]^2\right)}$ & $\substack{[i]^2\left([n][k-i+1]-[k]^2\right)}$ & $\substack{q^{2i+1}[k-i][n-k-i]\cdot\\ [i]^2\left([n][k-2]-[k]^2\right)}$ & $\substack{[i]^2\bigl(q^{k-2}[n]\left(2q[i-1]+q+1\right)+\\
    [i]^2\left([n][k-2]-[k]^2\right)\bigr)}$\\
    \end{tabular}
\end{center}
\end{lemma}
\; \\
\begin{theorem}
\label{innprodmat24}
Let $x,y\in X$ satisfy $1<\partial(x,y)<k$. In the following table, for each vector $u$ in the header column, and each vector $v$ in the header row, the $(u,v)$-entry of the table gives the inner product $\left<u,v\right>$. Write $i=\partial(x,y)$.
\begin{center}
    \begin{tabular}{c|c c c}
        $\left<\;,\;\right>$ & $A_{xy}^{+}$ & $A_{xy}^{0}$ & $A_{xy}^{-}$\\ 
         \hline
         \\
         $\widehat{x}$ & $\substack{q^{i+1}[i][n-k-i]\left([n][k-1]-[k]^2\right)}$ & $\substack{(q-1)[i]^2\left([n][k-1]-[k]^2\right)}$ & $\substack{q^{i+1}[i][k-i]\left([n][k-1]-[k]^2\right)}$\\
         \\
         $\widehat{y}$ & $\substack{q^{i+1}[i][n-k-i]\left([n][k-i]-[k]^2\right)}$ & $\substack{(q-1)[i]^2\left([n][k-i]-[k]^2\right)}$ & $\substack{q^{i+1}[i][k-i]\left([n][k-i]-[k]^2\right)}$\\
         \\
        $B_{xy}$ & $\substack{q^{2i+2}[i][k-i][n-k-i]\\
    \bigl(\left(q^i[n-k-i]-1\right)\left([n][k-2]-[k]^2\right)\\ 
    +\left([n][k-1]-[k]^2\right)\bigr)}$ & $\substack{q^{2i+1}\left(q-1\right)[k-i][n-k-i]\\ [i]^2\left([n][k-2]-[k]^2\right)}$ & $\substack{q^{2i+2}[i][k-i][n-k-i]\\
    \bigl(\left(q^i[k-i]-1\right)\left([n][k-2]-[k]^2\right)\\ 
    +\left([n][k-1]-[k]^2\right)\bigr)}$\\
    \\
    $C_{xy}$ & $\substack{q^{i+1}[n-k-i][i]^2\\ \bigl(q^{k-2}[n]+[i]\left([n][k-2]-[k]^2\right)\bigr)}$ & $\substack{\left(q-1\right)[i]^2\bigl(q^{k-2}[n]\left(2[i]-1\right)\\ +[i]^2\left([n][k-2]-[k]^2\right)\bigr)}$ & $\substack{q^{i+1}[k-i][i]^2\\ \bigl(q^{k-2}[n]+[i]\left([n][k-2]-[k]^2\right)\bigr)}$\\
    \end{tabular}
\end{center}
\end{theorem}

\begin{proof}
The entries in the first two rows are immediate from Theorem \ref{aprodmat21}. 

Next we calculate the inner product $\Bigl<B_{xy},A_{xy}^{+}\Bigr>$. 

    Using (\ref{aplin}), 
    \begin{equation*}
        \Bigl<B_{xy},A_{xy}^{+}\Bigr>=
        q^{i+1}[n-k-i][i-1]\bigl<B_{xy},\widehat{x}\bigr>+q^{2i}[n-k-i]\Bigl<B_{xy},\widehat{x\cap y}\Bigr>-[i]\Bigl<B_{xy},\widehat{x+y}\Bigr>.
    \end{equation*}
    In the above equation, evaluate the right-hand side using Lemma \ref{innprodmat23}.

    We have now calculated the inner product $\Bigl<B_{xy},A_{xy}^{+}\Bigr>$. For the other inner products the calculations are similar, and omitted.    
\end{proof}

\begin{theorem}
\label{innprodmat25}
Let $x,y\in X$ satisfy $1<\partial(x,y)<k$. In the following table, for each vector $u$ in the header column, and each vector $v$ in the header row, the $(u,v)$-entry of the table gives the inner product $\left<u,v\right>$. Write $i=\partial(x,y)$.
\begin{center}
    \begin{tabular}{c|c c c}
        $\left<\;,\;\right>$ & $A_{xy}^{+}$ & $A_{xy}^{0}$ & $A_{xy}^{-}$\\ 
         \hline
         \\
         $A_{xy}^{+}$ & $\substack{q^{2i+2}[i][n-k-i]
        \bigl(q^{k-i-2}[n][n-k]\\ +[i][n-k-i]\left([n][k-2]-[k]^2\right)\bigr)}$ & $\substack{q^{i+1}\left(q-1\right)[n-k-i][i]^2\\ \bigl(q^{k-2}[n]+[i]\left([n][k-2]-[k]^2\right)\bigr)}$ & $\substack{q^{2i+2}[k-i][n-k-i]\\ [i]^2\left([n][k-2]-[k]^2\right)}$ \\
        \\
        $A_{xy}^{0}$ & $\substack{q^{i+1}\left(q-1\right)[n-k-i][i]^2\\ \bigl(q^{k-2}[n]+[i]\left([n][k-2]-[k]^2\right)\bigr)}$& $\substack{\left(q-1\right)[i]^2\Bigl(q^{k-2}[n]\bigl(2\left(q-1\right)[i]+1\bigr)\\
        +\left(q-1\right)[i]^2\left([n][k-2]-[k]^2\right)\Bigr)}$ & $\substack{q^{i+1}(q-1)[k-i][i]^2\\ \bigl(q^{k-2}[n]+[i]\left([n][k-2]-[k]^2\right)\bigr)}$\\
        \\
        $A_{xy}^{-}$ & $\substack{q^{2i+2}[k-i][n-k-i]\\ [i]^2\left([n][k-2]-[k]^2\right)}$ & $\substack{q^{i+1}(q-1)[k-i][i]^2\\ \bigl(q^{k-2}[n]+[i]\left([n][k-2]-[k]^2\right)\bigr)}$ & $\substack{q^{2i+2}[i][k-i]
        \bigl(q^{k-i-2}[n][k]\\ 
        +[i][k-i]\left([n][k-2]-[k]^2\right)\bigr)}$\\
    \end{tabular}
\end{center}
\end{theorem}

\begin{proof}
    We will calculate the inner product $\Bigl<A_{xy}^{+},A_{xy}^{+}\Bigr>$. 

    Using (\ref{aplin}), 
    \begin{equation*}
        \Bigl<A_{xy}^{+},A_{xy}^{+}\Bigr>=
        q^{i+1}[n-k-i][i-1]\Bigl<A_{xy}^+,\widehat{x}\Bigr>+q^{2i}[n-k-i]\Bigl<A_{xy}^+,\widehat{x\cap y}\Bigr>-[i]\Bigl<A_{xy}^+,\widehat{x+y}\Bigr>.
    \end{equation*}
    In the above equation, evaluate the right-hand side using Theorem \ref{aprodmat21}.

    We have now calculated the inner product $\Bigl<A_{xy}^{+},A_{xy}^{+}\Bigr>$. For the other inner products the calculations are similar, and omitted.
\end{proof}

\section{Some combinatorics and algebra involving the $y$-partition of $\Gamma(x)$}

Let $x,y \in X$ satisfy $1 < \partial(x,y) < k$. In (\ref{5orbits}) we partitioned $\Gamma(x)$ into five orbits for $\Stab(x,y)$. In this section, we describe the edges between pairs of orbits in this partition. In this description we use a 5 by 5 matrix. We find the eigenvalues of this matrix. For each eigenvalue we display a row eigenvector and column eigenvector.

\begin{theorem}
\label{adjacency}
Let $x,y\in X$ satisfy $1<\partial(x,y)<k$, and consider the orbits of $\Stab(x,y)$ on $\Gamma(x)$. Referring to the table below, for each orbit $\mathcal{O}$ in the header column, and each orbit $\mathcal{N}$ in the header row, the $(\mathcal{O},\mathcal{N})$-entry gives the number of vertices in $\mathcal{N}$ that are adjacent to a given vertex in $\mathcal{O}$. Write $i=\partial(x,y)$.

\begin{center}
\begin{tabular}{c|c c c c c}
     & $\mathcal{B}_{xy}$ & $\mathcal{C}_{xy}$ & $\mathcal{A}_{xy}^{+}$ & $\mathcal{A}_{xy}^{0}$ & $\mathcal{A}_{xy}^{-}$\\
    \hline \\
    $\mathcal{B}_{xy}$ & $\substack{q^{i+1}[k-i]\\+q^{i+1}[n-k-i]-q-1}$ & $0$ & $q[i]$ & $0$ & $q[i]$\\
   \\
    $\mathcal{C}_{xy}$ & $0$ & $2q[i-1]$ & $q^{i+1}[n-k-i]$ & $(q-1)\bigl(2[i]-1\bigr)$ & $q^{i+1}[k-i]$\\
   \\
    $\mathcal{A}_{xy}^{+}$ & $q^{i+1}[k-i]$ & $[i]$ & $q[n-k]-q-1$ & $(q-1)[i]$ & $0$\\
    \\
    $\mathcal{A}_{xy}^{0}$ & $0$ & $2[i]-1$ & $q^{i+1}[n-k-i]$ & $(q-1)\bigl(2[i]-1\bigr)-1$ &$q^{i+1}[k-i]$\\
    \\
    $\mathcal{A}_{xy}^{-}$ & $q^{i+1}[n-k-i]$ & $[i]$ & $0$ & $(q-1)[i]$ & $q[k]-q-1$\\
\end{tabular}

\end{center}
\end{theorem}

\begin{proof}
We will verify the $\bigl(\mathcal{B}_{xy},\mathcal{B}_{xy}\bigr)$-entry of the table. Pick a vertex $w\in \mathcal{B}_{xy}$. Let $\#$ denote the number of vertices in $\mathcal{B}_{xy}$ that are adjacent to $w$. Note that $\#$ is independent of the choice of $w$, because the partition (\ref{5orbits}) is equitable.

We now compute $\#$. By construction, each vertex in $\mathcal{B}_{xy}$ is at distance at most $2$ from $w$. Using Lemma \ref{introlem1}(iii), (iv) we obtain

\begin{equation}
\label{wbxy}
\begin{aligned}
    \bigl<\widehat{w},B_{xy}\bigr>=\sum_{z\in \mathcal{B}_{xy}}\bigl<\widehat{w},\widehat{z}\bigr>=\;&q^k[k][n-k]+\#\Bigl([n][k-1]-[k]^2\Bigr)\\
    &\qquad \qquad+\Bigl(\bigl\vert \mathcal{B}_{xy}\bigr\vert-\#-1\Bigr)\Bigl([n][k-2]-[k]^2\Bigr).
\end{aligned}
\end{equation}

By construction,
\begin{equation}
    \label{bxybxy}
    \bigl<B_{xy},B_{xy}\bigr>=\bigl\vert \mathcal{B}_{xy}\bigr\vert\bigl<\widehat{w},B_{xy}\bigr>.
\end{equation}

We now evaluate (\ref{bxybxy}). The left-hand side is evaluated using the $\bigl(B_{xy},B_{xy}\bigr)$-entry in the table of Lemma \ref{innprodmat22}. The right-hand side is evaluated using (\ref{wbxy}) and $b_i=\bigl\vert\mathcal{B}_{xy}\bigr\vert$; the value of $b_i$ is given in (\ref{sizebc}).

After evaluating (\ref{bxybxy}), we solve the resulting equation for $\#$; this yields the $\bigl(\mathcal{B}_{xy},\mathcal{B}_{xy}\bigr)$-entry of the table. The other entries are obtained in a similar fashion.
\end{proof}

\begin{definition}
    For $1<i<k$ let $\mathcal{M}_i$ denote the $5\times 5$ matrix in Theorem \ref{adjacency}. 
\end{definition}

Note that $\mathcal{M}_i$ is not symmetric. We now give the transpose $\mathcal{M}_i^{t}$.

\begin{lemma}
    For $1<i<k$, we have $\mathcal{M}_i^{t}=D\mathcal{M}_iD^{-1}$, where $D={\rm{diag}}\bigl(b_i, c_i, a_i^{+}, a_i^{0}, a_i^{-}\bigr)$. Recall $b_i,c_i$ from (\ref{sizebc}) and $a_i^{+},a_i^{0},a_i^{-}$ from (\ref{acountdef}).
\end{lemma}

\begin{proof}
Immediate.    
\end{proof}
 
Our next goal is to find the eigenvalues of $\mathcal{M}_i$. For each eigenvalue we display a row eigenvector and a column eigenvector.

\begin{lemma}
    For $1<i<k$, the eigenvalues of the matrix $\mathcal{M}_i$ are
    \begin{equation*}
        a_1,\qquad q[n-k]-q-1,\qquad q[k]-q-1 ,\qquad -1 ,\qquad -q-1,
    \end{equation*}
    where $a_1=q[k]+q[n-k]-q-1$.
\end{lemma}

\begin{proof}
    Routine.
\end{proof}

\begin{lemma}
\label{eigentable}
    For $1<i<k$ we consider the matrix $\mathcal{M}_i$. In the table below, for each eigenvalue of $\mathcal{M}_i$, we display a corresponding row eigenvector and column eigenvector. Recall $b_i,c_i$ from (\ref{sizebc}) and $a_i^{+},a_i^{0},a_i^{-}$ from (\ref{acountdef}).

    \begin{center}
    \begin{tabular}{c|c|c}
        Eigenvalue of $\mathcal{M}_i$ & corresponding row eigenvector & corresponding column eigenvector\\ 
         \hline
         & & \\
         $a_1$ & $\bigl(b_i,c_i,a_i^{+},a_i^{0},a_i^{-}\bigr)$ &$(1,1,1,1,1)^{t}$\\
         & & \\
         $q[n-k]-q-1$ & $\bigl(a_i^{+},-c_i,-a_i^{+},-a_i^{0},qc_i\bigr)$ & $\bigl(qc_i,-a_i^{-},-a_i^{-},-a_i^{-},qc_i\bigr)^{t}$\\
         & & \\
         $q[k]-q-1$ & $\bigl(a_i^{-},-c_i,qc_i,-a_i^{0},-a_i^{-}\bigr)$& $\bigl(qc_i,-a_i^{+},qc_i,-a_i^{+},-a_i^{+}\bigr)^t$\\ 
         & & \\
         $-1$ & $(0,1,0,-1,0)$ & $(0,q-1,0,-1,0)^t$\\ 
         & & \\
         $-q-1$ & $(q,1,-q,q-1,-q)$ & $\bigl(qc_i,b_i,-a_i^{-},b_i,-a_i^{+}\bigr)^t$                
    \end{tabular}
\end{center}
\end{lemma}

\begin{proof}
    Routine.
\end{proof}

\begin{remark}{\cite{Liang, Watanabe}}
    For $x\in X$ the spectrum of the local graph $\Gamma(x)$ is given in the table below. Recall $a_1=q[k]+q[n-k]-q-1$.

    \begin{center}
    \begin{tabular}{c|c}
        Eigenvalue & Multiplicity\\ 
         \hline
         \\
         $a_1$ & $1$\\
         \\
         $q[n-k]-q-1$ & $[k]-1$\\
         \\
         $q[k]-q-1$ & $[n-k]-1$\\ 
         \\
         $-1$ & $(q-1)[k][n-k]$\\ 
         \\
         $-q-1$ & $q^2[k-1][n-k-1]$\\
    \end{tabular}
\end{center}
\end{remark}

\section*{Acknowledgement}
The author is currently a graduate student at the University of Wisconsin-Madison. He would like to thank his advisor, Paul Terwilliger, for many valuable ideas and suggestions for this paper.

Ian Seong\\
Department of Mathematics\\
University of Wisconsin \\
480 Lincoln Drive \\
Madison, WI 53706-1388 USA \\
email: iseong@wisc.edu\\
\end{document}